\documentclass{amsart}

\usepackage{amssymb}
\usepackage[hidelinks]{hyperref}
\usepackage[textsize=footnotesize]{todonotes}
\usepackage{tikz}

\theoremstyle{plain}
\newtheorem{theorem}{Theorem}[section]

\newtheorem{corollary}[theorem]{Corollary}
\newtheorem{proposition}[theorem]{Proposition}

\theoremstyle{definition}
\newtheorem{definition}[theorem]{Definition}

\newtheorem{example}[theorem]{Example}

\newtheorem{question}{Question}

\theoremstyle{remark}
\newtheorem*{remark}{Remark}

\newtheorem*{notation}{Notation}

\newcommand{\bN}{\mathbb{N}}
\newcommand{\N}{\bN}
\newcommand{\bZ}{\mathbb{Z}}
\newcommand{\Z}{\bZ}
\newcommand{\cA}{\mathcal{A}}

\newcommand{\cC}{\mathcal{C}}

\newcommand{\cI}{\mathcal{I}}
\newcommand{\I}{\cI}
\newcommand{\cJ}{\mathcal{J}}
\newcommand{\J}{\cJ}

\newcommand{\cP}{\mathcal{P}}

\newcommand{\cZ}{\mathcal{Z}}
\newcommand{\continuum}{\mathfrak{c}}
\newcommand{\fin}{\mathrm{Fin}}

\begin{document}


\title{Densities for sets of natural numbers vanishing on a given family}

\author{Rafa\l{} Filip\'{o}w}
\address[R.~Filip\'{o}w]{Institute of Mathematics\\ Faculty of Mathematics, Physics and Informatics\\ University of Gda\'{n}sk\\ ul. Wita Stwosza 57\\ 80-308 Gda\'{n}sk\\ Poland}
\email{Rafal.Filipow@ug.edu.pl}
\urladdr{http://mat.ug.edu.pl/~rfilipow}

\author{Jacek Tryba}
\address[J.~Tryba]{Institute of Mathematics\\ Faculty of Mathematics, Physics and Informatics\\ University of Gda\'{n}sk\\ ul. Wita Stwosza 57\\ 80-308 Gda\'{n}sk\\ Poland}
\email{Jacek.Tryba@ug.edu.pl}
\thanks{The second author has been supported by the grant MNiD-539-5100-B348-19.}

\date{\today}




\begin{abstract}
Abstract upper densities are monotone and subadditive functions from the power set of positive integers into the unit real interval that generalize the upper densities used in number theory, including the upper asymptotic density, the upper Banach density, and the upper logarithmic density.

At the open problem session of the Workshop ``Densities and their application'', held at St.~\'{E}tienne in July 2013, G.~Grekos asked a question whether there is 
a ``nice'' abstract upper density, whose the family of null sets is precisely a given ideal of subsets of $\N$,
where ``nice'' would mean the properties of the familiar densities consider in number theory.

In 2018, M.~Di Nasso and R.~Jin (Acta Arith. 185 (2018), no. 4) 
showed that the answer is positive for 
the summable ideals (for instance, the family of finite sets and the family of sequences whose series of reciprocals converge) when ``nice'' density means translation invariant and rich density (i.e.~density which is \emph{onto} the unit interval).

In this paper we extend their result to all ideals with the Baire property. 
\end{abstract}


\maketitle


\setcounter{tocdepth}{1}
\tableofcontents


\section{Abstract densities and ideals}
\label{sec:AbstractDensitiesIdeals}

\begin{definition}
\label{def:abstract-upper-density}
An \emph{abstract upper density} on $\N$ is a function $\delta:\cP(\N)\to[0,1]$ that satisfies the following properties:
\begin{enumerate}
\item $\delta(\N)=1$,\label{def:abstract-upper-density-normed}
\item if $F\subseteq\N$ is finite then $\delta(F)=0$,\label{def:abstract-upper-density-vanishes-on-FIN}
\item if $A\subseteq B$ then  $\delta(A)\leq\delta(B)$,\label{def:abstract-upper-density-monotone}
\item $\delta(A\cup B)\leq \delta(A)+\delta(B)$.\label{def:abstract-upper-density-subadditive}
\end{enumerate}
\end{definition}

\begin{notation}
For $A\subseteq \N$ and $k\in \Z$ we write $A+k$ to denote the set $(A+k)\cap \N$. 
\end{notation}

\begin{definition}
An abstract upper density $\delta$ is 
\begin{enumerate}
\item \emph{translation invariant} if $\delta(A+k)=\delta(A)$ for every $A\subseteq\N$ and $k\in \Z$;
\item \emph{rich}  if for every $r\in [0,1]$ there is $A\subseteq \N$ with $\delta(A)=r$.
\end{enumerate}
\end{definition}

\begin{definition}
An \emph{ideal on $\N$} (in short \emph{ideal}) is a  family $\I\subseteq\cP(\N)$ that satisfies the following properties:
\begin{enumerate}
\item if $A,B\in \I$ then $A\cup B\in\I$,
\item if $A\subseteq B$ and $B\in\I$ then $A\in\I$,
\item $\I$ contains all finite subsets of $\N$,
\item $\N\notin\I$.
\end{enumerate}
An ideal $\I$ is \emph{translation invariant} if $A+k\in \I$ for every $A\in \I$ and $k\in \Z$.
For an ideal $\I$, 
we write $\I^+=\{A\subseteq \N: A\notin\I\}$ and call it the \emph{coideal of $\I$}, 
and we write $\I^*=\{A\subseteq \N: \N\setminus A\in\I\}$ and call it the \emph{dual filter of $\I$}.
\end{definition}

\begin{remark}
We also consider ideals on any infinite countable set by identifying this set with $\N$ via a fixed bijection.
\end{remark}

\begin{example}\
\begin{enumerate}
\item  The family $\fin = \{A\subseteq\N: A\text{ is finite}\}$ is a translation invariant  ideal.
\item Let $f:\N\to[0,\infty)$ be a~non-increasing function such that $\sum_{n=1}^\infty f(n)=\infty$. 
The family $\I_f = \{A\subseteq\N: \sum_{n\in A}f(n)<\infty\}$
is a translation invariant  ideal. We call it the \emph{summable ideal} determined by $f$.
\item The family $\I_d=\{A\subseteq\N: \limsup_{n\to\infty}|A\cap \{1,\dots,n\}|/n=0\}$ of all sets of asymptotic density zero 
is a translation invariant  ideal. We call it the \emph{asymptotic density zero ideal}.
\end{enumerate}
\end{example}

\begin{example}\
\label{exm:denisties-ideals}
\begin{enumerate}
\item Let $\delta$ be an abstract upper density. The family 
$\cZ_\delta=\{A\subseteq\N: \delta(A)=0\}$ is an ideal.
If $\delta$ is translation invariant, so is the ideal $\cZ_\delta$.
\item Let $\I$ be an ideal.
The function $\delta:\cP(\N)\to[0,1]$ given by 
$$\delta(A) = \begin{cases}
0&\text{if $A\in \I$},\\
1 & \text{otherwise}
\end{cases}$$
is an abstract upper density and $\I=\cZ_\delta$. 
If $\I$ is translation invariant, so is $\delta$.\label{exm:denisties-ideals:always-exists}
\end{enumerate}
\end{example}

An ideal $\I$ is \emph{maximal} if $\I\subseteq\J$ implies $\I=\J$ for every ideal $\J$ (i.e.~$\I$ is a maximal element with respect to inclusion among all ideals).
It is known that an ideal $\I$ is maximal 
if and only if $A\in \I$ or $\N\setminus A\in \I$ for every $A\subseteq\N$
(see e.g.~\cite[Lemma~7.4, p.~74]{MR1940513}).

A translation invariant  ideal $\I$ is \emph{T-maximal} if $\I\subseteq\J$ implies $\I=\J$ for every translation invariant ideal $\J$ (i.e.~$\I$ is a maximal element with respect to inclusion among all translation invariant ideals).

\begin{remark}
By identifying sets of natural numbers with their characteristic functions,
we equip $\cP(\N)$ with the topology of the Cantor space $\{0,1\}^\N$ and therefore
we can assign topological complexity to ideals on $\N$.
In particular, an ideal $\I$ has the Baire property if $\I$ has the Baire property as a subset of the Cantor space. (Recall that a subset $B$ of a topological space $X$ has the \emph{Baire property} if there is an open set $U\subseteq X$ such that the set $(B\setminus U)\cup(U\setminus B)$ is of first category.)
\end{remark}

\begin{theorem}[{Sierpi\'{n}ski, see e.g.~\cite[Theorem~4.1.1]{MR1350295}}]
\label{thm:BP-ideal-iff-meager}
Let $\I$ be an ideal on $\N$.
Then $\I$ is of first category or $\I$ does not have the Baire property.
\end{theorem}

\begin{theorem}[{Talagrand~\cite[Th\'{e}or\`{e}me 21]{MR579439} (see also \cite[Theorem 4.1.2]{MR1350295})}]
\label{thm:talagrand-characterization}
An ideal $\I$ on $\N$ has the Baire property if and only if there is an increasing sequence $k_1<k_2<\dots$ such that 
if there are infinitely many $n$ with $[k_n,k_{n+1})\cap\N\subseteq A$, then $A\in \I^+$.
\end{theorem}

\begin{example}\ 
\label{example:ideals-with(out)-BP}
\begin{enumerate}
\item 
It is known that the ideals $\fin$, $\I_f$ and $\I_d$ have the Baire property, in fact they are Borel ideals (see e.g.~\cite{MR1711328}).
\item 
Maximal ideals do not have the Baire property (see e.g.~\cite[Theorem~4.1.1]{MR1350295}).
\end{enumerate}
\end{example}

\begin{theorem}[{Plewik \cite[Theorem~1]{MR1122275}}]
\label{thm:plewik-countable-intersection-of-nonBP-ideals}
The intersection of a countable family of ideals without the Baire property is an ideal without the Baire property.
In particular, the intersection of countable family of maximal ideals is an ideal without the Baire property.
\end{theorem}

\begin{definition}
Let $\I$ be an ideal on $\N$. We say that a family $\cA\subseteq\cP(\N)$ is 
\begin{enumerate}
\item 
\emph{$\I$ almost disjoint family}
(in short \emph{$\I$-AD family}) if
\begin{enumerate}
\item $\cA\subseteq\I^+$,
\item $A\cap B\in \I$ for any distinct $A,B\in \cA$.
\end{enumerate}
\item 
\emph{$\I$ translation almost disjoint family}
(in short \emph{$\I$-TAD family}) if
\begin{enumerate}
\item $\cA\subseteq\I^+$,
\item $A\cap (B+k)\in \I$ for any distinct $A,B\in \cA$ and $k\in \Z$.
\end{enumerate}
\end{enumerate}
\end{definition}


\section{Rich densities}

\begin{theorem}
\label{thm:I-AD-of-cardinality-continuum-implies-abstract-upper-density}
Let $\I$ be an ideal on $\N$.
If there exists an $\I$-AD family of cardinality $\continuum$, then there is a rich abstract upper density $\delta$ such that $\cZ_\delta=\I$.
\end{theorem}

\begin{proof}
Using Zorn's lemma, it is not difficult to extend $\I$-AD family to a maximal (with respect to the inclusion) $\I$-AD family.
Let $\cA$ be a maximal $\I$-AD family of cardinality $\continuum$.
Let $\cA=\{A_\alpha:\alpha<\continuum\}$ and $(0,1)=\{r_\alpha:\alpha<\continuum\}$.
We define $\delta:\cP(\N)\to[0,1]$ by 
$$\delta(A) = \sup\{r_\alpha: A_\alpha\cap A\in \I^+\}$$
with the convention that $\sup\emptyset=0$.

\smallskip

First, we show that $\delta$ is an abstract upper density.

For every $\alpha<\continuum$ we have $A_\alpha\cap \N=A_\alpha\in \I^+$, so $\delta(\N) = \sup (0,1) = 1$.

For every $\alpha<\continuum$ and a finite set $F\subseteq\N$, we have $A_\alpha\cap F \subseteq F \in \I$, so $\delta(F) = \sup \emptyset  = 0$.

If $A\subseteq B$, then $A_\alpha\cap A\in\I^+$ implies $A_\alpha\cap B\in\I^+$. Thus $\delta(A)\leq \delta(B)$.

If $A_\alpha\cap (A\cup B)\in\I^+$, then 
$A_\alpha\cap A\in\I^+$
or
$A_\alpha\cap B\in\I^+$.
Thus 
$\delta(A\cup B)\leq \delta(A)+\delta(B)$.

\smallskip

Now we show that $\delta$ is rich.

It is enough to show that $\delta(A_\alpha)=r_\alpha$ for every $\alpha<\continuum$.
Since $\cA$ is an $\I$-AD family, $A_\beta\cap A_\alpha\in \I$ for every $\alpha\neq\beta$. On the other hand, $A_\alpha\cap A_\alpha = A_\alpha\in\I^+$.
Thus $\delta(A_\alpha) = \sup\{r_\alpha\}=r_\alpha$.

\smallskip

Finally, we show that $\cZ_\delta=\I$.

If $A\in \I$, then $A_\alpha\cap A \subseteq A\in \I$ for every $\alpha<\continuum$. Thus $\delta(A) = \sup\emptyset=0$.

If $A\notin \I$, then, by maximality of $\cA$, there is $\alpha<\continuum$ with $A_\alpha\cap A\in\I^+$. Thus $\delta(A) \geq r_\alpha>0$.
\end{proof}

\begin{theorem}
\label{thm:rich-abstract-upper-density}
Let $\I$ be an ideal on $\N$ with the Baire property.
There is a rich abstract upper density $\delta$ such that $\cZ_\delta=\I$.
\end{theorem}

\begin{proof}
Since it is known (see \cite{MR3800755} or \cite{MR3836186}) that there  is an $\I$-AD family of cardinality $\continuum$ for every ideal with the Baire property, Theorem~\ref{thm:I-AD-of-cardinality-continuum-implies-abstract-upper-density} finishes the proof.
\end{proof}

\begin{definition}[{Oliver~\cite{MR2078923}}]
Let $\I$ be an ideal on $\N$.
We say that a partition $\{P_n: n \in \N\}$ of $\N$ is an \emph{$\omega$-partition of $\N$ with respect to $\I$} if
\begin{enumerate}
\item $P_n\not\in\I$ for all $n\in\N$,
\item for every $A\subseteq\N$, if $A\cap P_n\in\I$ for every $n\in\N$, then $A\in\I$.
\end{enumerate}
\end{definition}

\begin{theorem}
\label{thm:omega-partition-implies-rich-density}
If there exists an $\omega$-partition $\{P_n: n \in \N\}$ of $\N$ with respect to $\I$, then  there is a rich abstract upper density $\delta$ with $\cZ_\delta = \I$.
\end{theorem}
\begin{proof}
It is easy to see that 
$$\I = \{A\subseteq\N: A\cap P_n\in \I\text{ for every $n\in\N$}\}$$
and
$$\delta(A) = \sum_{A\cap P_n\notin\I}\frac{1}{2^n}$$
is an abstract upper density with $\cZ_\delta=\I$.

To see that $\delta$ is rich it is enough to notice that for every $r\in (0,1)$ there is $B\subseteq \N$ with $\sum_{n\in B}1/2^n=r$ and then $\delta(\bigcup\{P_n: n\in B\})=r$.
\end{proof}

\begin{theorem}
\label{thm:nonAD-ideal-with-rich-density}
There exists an ideal $\I$ for which there is a rich abstract upper density $\delta$ with $\cZ_\delta = \I$, but there is no $\I$-AD family of cardinality $\continuum$.
\end{theorem}

\begin{proof}
Let $\{P_n:n\in\N\}$ be a partition of $\N$ such that $P_n$ is infinite for all $n$.
Let $\I_n$ be a maximal ideal on $P_n$ for every $n\in \N$.
It is easy to see that 
$$\I = \{A\subseteq\N: A\cap P_n\in \I_n\text{ for every $n\in\N$}\}$$
is an ideal on $\N$.
Since $\{P_n:n\in \N\}$ is an $\omega$-partition of $\N$ with respect to $\I$, there is a rich abstract upper density $\delta$ with $\cZ_\delta=\I$ (by Theorem~\ref{thm:omega-partition-implies-rich-density}).

Lastly, we show that there is no $\I$-AD family of cardinality $\continuum$ (in fact we show that there is no uncountable $\I$-AD family).
Let $\cA\subseteq\I^+$ be uncountable. For every $A\in \cA$ there is at least one $n_A\in \N$ with $A\cap P_{n_A}\in \I_{n_A}^+$. Since $\cA$ is uncountable, there are  distinct $A,B\in \cA$ with $n_A=n_B$.
Let $n=n_A$.
Since $\I_n$ is a maximal ideal on $P_n$, 
$\I_n^+=\I_n^*$.
Thus, $(A\cap  P_n)\cap (B\cap P_n)\in \I_n^*$, so $(A\cap B)\cap P_n\in \I_n^*=\I_n^+$.
Hence $A\cap B\notin \I$, so $\cA$ is not an $\I$-AD family.
\end{proof}

\begin{theorem}
\label{thm:AD-nonBP-ideal-with-rich-density}
There exists an ideal $\I$ without the Baire property for which there is $\I$-AD family of cardinality $\continuum$. 
\end{theorem}

\begin{proof}
Let $\{P_n:n\in\N\}$ be a partition of $\N$ such that $P_n$ is infinite for all $n$.
Let $\I_n$ be a maximal ideal on $P_n$ for every $n\in \N$.
It is easy to see that 
$$\I = \{A\subseteq\N: A\cap P_n\in \I_n\text{ for all but finitely many $n\in\N$}\}$$
is an ideal on $\N$.

Let $\cA$ be $\fin$-AD family of cardinality $\continuum$ (there is one, see e.g.~\cite[Lemma 9.21]{MR1940513}).
For $A\in \cA$, we define $C_A=\bigcup\{P_n:n\in A\}$.
Then $C_A\notin\I$ for every $A\in \cA$, and for distinct $A,B\in \cA$ we have $C_A\cap C_B=\bigcup\{P_n:n\in A\cap B\}\in \I$.
Thus, $\{C_A:A\in \cA\}$ is $\I$-AD family of cardinality $\continuum$.

Lastly, we show that $\I$ does not have the Baire property.
For every $n\in \N$ we define $\J_n = \{A\subseteq \N: A\cap P_n\in \I_n\}$.
It is easy to see that $\J_n$ is a maximal ideal on $\N$ for every $n\in \N$ and 
$$\I = \bigcup_{n\in\N}\bigcap_{k\geq n} \J_k.$$
By Theorem~\ref{thm:plewik-countable-intersection-of-nonBP-ideals}, 
the ideal $\bigcap_{k\geq n} \J_k$ does not have the Baire property for every $k\in \N$.
In particular, $\bigcap_{k\geq 1} \J_k$ does not have the Baire property, so it is of second category. Thus, $\I$ is of second category as a superset of $\bigcap_{k\geq 1} \J_k$.
Then, by Theorem~\ref{thm:BP-ideal-iff-meager}, $\I$ does not have the Baire property.
\end{proof}

\begin{theorem}
\label{thm:maximl-no-rich-aud}
Let  $\I$ be a maximal ideal.
If $\delta$ is an abstract upper density $\delta$ with $\cZ_\delta=\I$ then $\delta$ takes only two values: zero and one. In particular, there is no rich abstract upper density $\delta$ with $\cZ_\delta=\I$.
\end{theorem}

\begin{proof}
Let $A\subseteq\N$. 
If $A\in \I$ then $\delta(A)=0$. 
If $A\notin\I$, then $\N\setminus A\in \I$, hence 
$1=\delta(\N)\leq \delta(A)+\delta(\N\setminus A) = \delta(A)+0=\delta(A)\leq 1$, thus $\delta(A)=1$.
\end{proof}

\begin{theorem}
\label{thm:non-maximl-non-BP-no-rich-aud}
There exists a non-maximal ideal $\I$ 
for which there is no rich abstract upper density $\delta$ with $\cZ_\delta = \I$.
\end{theorem}

\begin{proof}
Let $N\geq 2$.
Let $\{P_n:n\leq N\}$ be a partition of $\N$ such that $P_n$ is infinite for all $n\leq N$.
Let $\I_n$ be a maximal ideal on $P_n$ for every $n\in \N$.
It is easy to see that 
$$\I = \{A\subseteq\N: A\cap P_n\in \I_n\text{ for every $n\leq N$}\}$$
is an ideal on $\N$. Since both $P_1\notin \I$ and $\N\setminus P_1\notin \I$, $\I$ is not a maximal ideal.

Let $\delta$ be any abstract upper density such that $\cZ_\delta=\I$.

Let $n\leq N$ and $A\subseteq P_n$.
If $A\in \I_n$ then $\delta(A)=0$. 
If $A\notin\I_n$, then $P_n \setminus A\in \I_n$, hence 
$\delta(P_n)\leq \delta(A)+\delta(P_n\setminus A) = \delta(A)+0=\delta(A)\leq \delta(P_n)$, hence $\delta(A)=\delta(P_n)$.

Let $A\notin\I$. There is at least one $n\leq N$ with $A\cap P_n\notin\I_n$, so $\delta(A)\geq \delta(A\cap P_n) = \delta(P_n)$.

Thus, $\delta(A)\geq \min\{\delta(P_n):n\leq N\}>0$ for every $A\notin\I$, hence $\delta$ is not rich.
\end{proof}

\bigskip

The following diagram summarizes all known relationships between ideals $\I$ and existence of rich abstract upper density with $\I=\cZ_\delta$. 

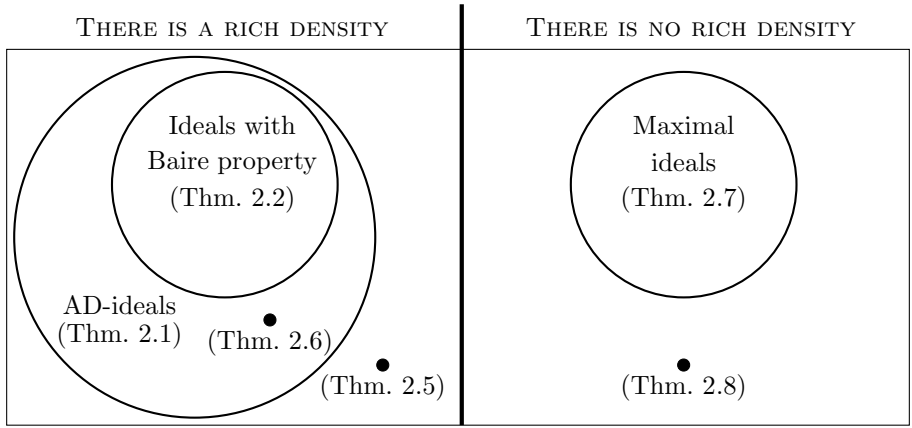
\begin{center}
\begin{tikzpicture}
\draw (-6,-2.5) rectangle (6,2.5);

\draw[ultra thick] (0.05,-2.5) -- (0.05,3.1);
\node at (-3,2.8) {\textsc{There is a rich density}};
\node at (3.1,2.8) {\textsc{There is no rich density}};

\draw[thick]  (-3.5,0) circle (2.4cm);
\node at (-4.5,-0.9) {AD-ideals};
\node at (-4.5,-1.3) {(Thm.~\ref{thm:I-AD-of-cardinality-continuum-implies-abstract-upper-density})};

\draw[thick]  (-3.1,0.7) circle (1.5cm);
\node at (-3,1.5) {Ideals with };
\node at (-3,1) {Baire property};
\node at (-3,0.5) {(Thm.~\ref{thm:rich-abstract-upper-density})};

\draw[fill] (-2.5,-1.1) circle (0.08cm);
\node at (-2.5,-1.4) {(Thm.~\ref{thm:AD-nonBP-ideal-with-rich-density})};

\draw[fill] (-1,-1.7) circle (0.08cm);
\node at (-1,-2) {(Thm.~\ref{thm:nonAD-ideal-with-rich-density})};

\draw[thick]  (3,0.7) circle (1.5cm);
\node at (3,1.5) {Maximal};
\node at (3,1) {ideals};
\node at (3,0.5) {(Thm.~\ref{thm:maximl-no-rich-aud})};

\draw[fill] (3,-1.7) circle (0.08cm);
\node at (3,-2) {(Thm.~\ref{thm:non-maximl-non-BP-no-rich-aud})};

\end{tikzpicture}
\end{center}


\section{Rich and translation  invariant densities}

\begin{theorem}
\label{thm:I-TAD-of-cardinality-continuum-implies-abstract-upper-density}
Let $\I$ be a translation invariant ideal on $\N$.
If there exists an $\I$-TAD family of cardinality $\continuum$, then there is a translation invariant and rich abstract upper density $\delta$ such that $\cZ_\delta=\I$.
\end{theorem}

\begin{proof}
Using Zorn's lemma, it is not difficult to extend $\I$-TAD family to a maximal (with respect to the inclusion) $\I$-TAD family.
Let $\cA$ be a maximal $\I$-TAD family of cardinality $\continuum$.
Let $\cA=\{A_\alpha:\alpha<\continuum\}$ and $(0,1)=\{r_\alpha:\alpha<\continuum\}$.
We define $\delta:\cP(\N)\to[0,1]$ by 
$$\delta(A) = \sup\{r_\alpha: A_\alpha\cap (A+k)\in \I^+\text{ for some $k\in \Z$}\}$$
with the convention that $\sup\emptyset=0$.

\smallskip

First, we show that $\delta$ is an abstract upper density.

For every $\alpha<\continuum$ we have $A_\alpha\cap (\N+0)=A_\alpha\in \I^+$, so $\delta(\N) = \sup (0,1) = 1$.

For every $\alpha<\continuum$, $k\in \Z$ and a finite set $F\subseteq\N$, we have $A_\alpha\cap (F+k) \subseteq F+k \in \I$, so $\delta(F) = \sup \emptyset  = 0$.

If $A\subseteq B$, then $A_\alpha\cap (A+k)\in\I^+$ implies $A_\alpha\cap (B+k)\in\I^+$. Thus $\delta(A)\leq \delta(B)$.

If $A_\alpha\cap ((A\cup B)+k)\in\I^+$, then 
$A_\alpha\cap (A+k)\in\I^+$
or
$A_\alpha\cap (B+k)\in\I^+$.
Thus 
$\delta(A\cup B)\leq \delta(A)+\delta(B)$.

\smallskip

Second, we show that $\delta$ is translation invariant.

If $A_\alpha\cap ((A+m)+k)\in\I^+$ then 
$A_\alpha\cap (A+(m+k))\in\I^+$ and vice versa. Thus $\delta(A+m) = \delta(A)$.

\smallskip

Now we show that $\delta$ is rich.

It is enough to show that $\delta(A_\alpha)=r_\alpha$ for every $\alpha<\continuum$.
Since $\cA$ is an $\I$-TAD family, $A_\beta\cap (A_\alpha+k)\in \I$ for every $k\in \Z$ and $\alpha\neq\beta$. On the other hand, $A_\alpha\cap (A_\alpha+0) = A_\alpha\in\I^+$.
Thus $\delta(A_\alpha) = \sup\{r_\alpha\}=r_\alpha$.

\smallskip

Finally, we show that $\cZ_\delta=\I$.

If $A\in \I$, then $A_\alpha\cap (A+k) \subseteq A+k\in \I$ for every $\alpha<\continuum$ and $k\in \Z$. Thus $\delta(A) = \sup\emptyset=0$.

If $A\notin \I$, then, by maximality of $\cA$, there is $\alpha<\continuum$ and $k\in \Z$ with $A_\alpha\cap (A+k)\in\I^+$. Thus $\delta(A) \geq r_\alpha>0$.
\end{proof}

\begin{theorem}
\label{thm:I-TAD-of-cardinality-continuum}
Let $\I$ be an ideal on $\N$ with the Baire property.
There is an $\I$-TAD family of cardinality $\continuum$.
\end{theorem}

\begin{proof}
By Theorem~\ref{thm:talagrand-characterization} there is 
an increasing sequence $k_1<k_2<\dots$ such that 
if there are infinitely many $n$ with $[k_n,k_{n+1})\cap\N\subseteq A$, then $A\in \I^+$.
Without loss of generality we can assume that 
$\lim_{n\to\infty} (k_{n+1}-k_n)=\infty$.

Let us denote $I_{n} = [k_n,k_{n+1})\cap\N$ and $|I_n| = k_{n+1}-k_n$ for every $n\in\N$.

Let $\cA\subseteq \cP(\N)$ be a $\fin$-AD family of cardinality $\continuum$.
For every $A\in \cA$ we define 
$$C_A = \bigcup_{n\in A} I_{2n}.$$
We claim that $\cC = \{C_A:A\in \cA\}$ is the required family.

\smallskip

First, we show that $C_A\in \I^+$ for every $A\in\cA$.

If $A\in \cA$, then $A$ is infinite, so $C_A$ contains infinitely many intervals $[k_n,k_{n+1})\cap\N$, thus $C_A\in\I^+$.

\smallskip

Second, we show that $\cC$ is of cardinality $\continuum$.

If $A,B\in \cA$ are distinct, then there is $n\in (A\setminus B)\cup(B\setminus A)$, so $I_{2n}\subseteq (C_A\setminus C_B)\cup (C_B\setminus C_A)$.
Thus $|\cC|=|\cA|=\continuum$.

\smallskip

Finally, we show that $\cC$ is $\I$-TAD.

Let $A,B\in \cA$ be distinct and $k\in \Z$.
We have to show that $C_A\cap (C_B+k)\in \I$.

Let $N$ be such that $|I_n|>|k|$ for every $n>N$.
Then $I_{2n}+k\subseteq I_{2n-1}\cup I_{2n}\cup I_{2n+1}$ for every $n>N$.
Hence $(I_{2n}+k)\cap I_{2m}=\emptyset$ for any $n>N$, $m\in\N$, $n\neq m$.

Since 
$$
C_A\cap (C_B+k) 
= 
\bigcup_{m\in A} I_{2m}\cap\left(\bigcup_{n\in B} I_{2n}+k\right)
=
\bigcup_{m\in A} \bigcup_{n\in B} (I_{2m}\cap (I_{2n}+k))
$$
and $A\cap B$ is finite, we obtain that $C_A\cap (C_B+k)$ is contained in finitely many intervals $I_{2m}$. Thus $C_A\cap (C_B+k)\in \fin \subseteq \I$.
\end{proof}

\begin{theorem}
\label{thm:TA-rich-abstract-upper-density}
Let $\I$ be a translation invariant ideal on $\N$ with the Baire property.
There is a translation invariant and rich abstract upper density $\delta$ such that $\cZ_\delta=\I$.
\end{theorem}

\begin{proof}
Follows from Theorems~\ref{thm:I-TAD-of-cardinality-continuum-implies-abstract-upper-density} and \ref{thm:I-TAD-of-cardinality-continuum}.
\end{proof}

\begin{corollary}[{Di Nasso and Jin~\cite[Theorems~2.1 and 2.2]{MR3863054}}]
Let $\I$ be a summable ideal.
There is a translation invariant and rich abstract upper density $\delta$ such that $\cZ_\delta=\I$.
\end{corollary}

\begin{proof}
It is not difficult to check that summable ideals are translation invariant, and it is known that summable ideals have the Baire property (see e.g.~\cite[Example~1.2.3(c)]{MR1711328}). Thus Theorem~\ref{thm:TA-rich-abstract-upper-density} finishes the proof.
\end{proof}

\begin{question}
\label{q:TAD-ideal-withoutBP}
Does there exist a translation invariant  ideal $\I$ without the Baire property such that 
there is an $\I$-TAD family of cardinality $\continuum$?
\end{question}

\begin{definition}
Let $\I$ be an ideal on $\N$. We say that a sequence of reals $\langle x_n\rangle$ is \emph{$\I$-convergent to $L$} if 
$\{n\in\N: |x_n-L|\geq \varepsilon\}\in\I$
for every $\varepsilon>0$. The number $L$ is called \emph{an $\I$-limit of $\langle x_n\rangle$} 
and we denote it by $\lim^{\I} x_n$ .
\end{definition}

\begin{theorem}
\label{thm:nonTAD-ideal-with-rich-TA-density}
There exists a translation invariant  ideal $\I$ for which there is a rich translation invariant abstract upper density $\delta$ with $\cZ_\delta = \I$, but there is no $\I$-TAD family of cardinality $\continuum$.
\end{theorem}

\begin{proof}
Let $\J$ be a maximal ideal on $\N$.
It is known  (see e.g.~\cite[p.~3314]{MR1845008}) 
that 
$$\delta(A) =  {\lim}^{\J} \frac{|A\cap\{1,\dots,n\}|}{n}$$
is a translation invariant finitely additive measure (in particular, $\delta$ is a translation invariant abstract upper density).

Moreover, $\delta$ is rich. Indeed, it is enough to observe (see e.g.~\cite[p.~216]{MR2040266}) that for every $r\in [0,1]$ there is a set $A\subseteq \N$ such that the asymptotic density of $A$ equals $r$ i.e.~$r = \lim_{n\to\infty} \frac{|A\cap \{1,\dots,n\}|}{n}$.

Let $\I=\cZ_\delta$. Since $\delta$ is a translation invariant abstract upper density, $\I$ is a translation invariant ideal.  
To finish the proof we have to show that there is no $\I$-AD family of cardinality $\continuum$ (in fact we show that there is no uncountable $\I$-AD family).
Suppose $\cA\subseteq\I^+$ is an uncountable $\I$-AD family. 
For every $A\in \cA$, let  $n_A\in \N$ be such that $\delta(A)\geq 1/n_A$. 
Since $\cA$ is uncountable, there are infinitely many sets $A_1,A_2,\dots\in \cA$ with $n_{A_i}=n_{A_j}$ for every $i,j\in\N$.
Let $n=n_{A_1}$.
Since $A_j\cap A_j\in \I$ for distinct $i,j\in \N$, $\delta(A_i\cap A_j)=0$.
Since $\delta$ is a finitely additive measure,
$\delta(A_i\cup A_j) = \delta(A_i)+\delta(A_j)-\delta(A_i\cap A_j) = \delta(A_i)+\delta(A_j)$ for distinct $i,j\in \N$.
Finally,
$1\geq \delta(A_1\cup\dots \cup A_{n+1})= \delta(A_1)+\dots+\delta(A_{n+1}) \geq (n+1)/n>1$, a contradiction.
\end{proof}

We do not know if there is any translation invariant ideal such that there is no rich translation invariant abstract upper density $\delta$ with $\cZ_\delta = \I$.

\begin{question}
\label{q:Tideal-without-density}
Does there exist a translation invariant  ideal $\I$ such that 
there is no rich translation invariant abstract upper density $\delta$ with $\cZ_\delta = \I$?
\end{question}

Keeping Theorem~\ref{thm:maximl-no-rich-aud} in mind, it is natural to examine T-maximal ideals in  the quest for a solution of the above question. 

\begin{question}
Does there exist a T-maximal ideal $\I$ such there is a rich translation invariant abstract upper density $\delta$ with $\cZ_\delta = \I$?
\end{question}

In Theorems~\ref{thm:T-max-is-not-TADideal} and \ref{thm:T-max-is-not-maximal-ideal-limit}, 
we show that if the above question had positive answer, then the required density would be different from densities constructed in the proofs of Theorems~\ref{thm:I-TAD-of-cardinality-continuum-implies-abstract-upper-density} and \ref{thm:nonTAD-ideal-with-rich-TA-density}.

\begin{theorem}
\label{thm:T-max-is-not-TADideal}
If $\I$ is a T-maximal  ideal, then for any $A,B\notin \I$ there is $k$ with $(A+k)\cap B\notin \I$.
In particular, there is no $\I$-TAD family of cardinality $\continuum$.
\end{theorem}

\begin{proof}
Let $A,B\notin \I$.  Since $\I$ is maximal among all translation invariant ideals, there are $k_1,\dots,k_n\in \Z$ with $(A+k_1)\cup \dots\cup(A+k_n)\in \I^*$.
Thus, $B\cap [(A+k_1)\cup \dots\cup(A+k_n)]\notin \I$.
Hence, there is $i\leq n$ with $B\cap (A+k_i)\notin \I$.
\end{proof}

\begin{proposition}
\label{prop:intersection-of-maximal-ideals}
Let $\I$ be an ideal on $\N$.
\begin{enumerate}
\item The family 
$$t(\cI) =\{A\subseteq\N: A+k\in \cI\text{ for every $k\in \Z$}\}$$
is a translation invariant ideal.
\item 
If $\I$ does not  have the Baire property then $t(\I)$ does not have the Baire property. \label{prop:intersection-of-maximal-ideals-BP}
\end{enumerate}
\end{proposition}

\begin{proof}
(1)
Straightforward.
(\ref{prop:intersection-of-maximal-ideals-BP}) 
It follows from Theorem~\ref{thm:plewik-countable-intersection-of-nonBP-ideals}, because
$\I+k =\{A\subseteq\N:A-k\in\I\}$
is the ideal without the Baire property for every $k\in \Z$ 
and $t(\I) = \bigcap\{\I+k:k\in \Z\}$.
\end{proof}

\begin{theorem}
\label{thm:ideal-lim-density-is-not-t-of-max}
Let $\delta$ be defined as in the proof of Theorem~\ref{thm:nonTAD-ideal-with-rich-TA-density}.
Then $\cZ_\delta \neq t(\I)$ for any maximal ideal $\I$.
\end{theorem}

\begin{proof}
Suppose that $\cZ_\delta = t(\I)$ for some maximal ideal $\I$. Since $t(\I)\subseteq \I$ and one could easily notice that $\I_d\subseteq\cZ_\delta$, we get $\I_d\subseteq\I$. 
Let $A=\bigcup_{n\in\N}I_{2n}$, where $I_n=(2^{n},2^{n+1}]\cap \N$. Then $\limsup_{n\to\infty}\frac{|A\cap\{1,\dots,n\}|}{n}=\frac{2}{3}$ and $\liminf_{n\to\infty}\frac{|A\cap\{1,\dots,n\}|}{n}=\frac{1}{3}$ as $\frac{|I_{2n}|}{|I_{2n-1}\cup I_{2n}|}=\frac{2}{3}$ and $\frac{|I_{2n}|}{|I_{2n}\cup I_{2n+1}|}=\frac{1}{3}$. 
Therefore, $\delta(A)\in\left[\frac{1}{3},\frac{2}{3}\right]$, thus $A\not\in\cZ_\delta$.

Now, since $\I$ is a maximal ideal, either $A\in\I$ or $\N\setminus A \in\I$. Assume that $A\in\I$ (the other case is identical since $\N\setminus A\not\in\cZ_\delta$ too). By the supposition that $\cZ_\delta = t(\I)$, we know that $A\not\in t(\I)$, so there exists $k\in\Z$ such that $A+k\not\in\I$. It follows that $B=A+k\setminus A\not\in\I$. However, for all but finitely many $n\in\N$ we have 
$$\frac{|B\cap I_n|}{|I_n|}\leq \frac{k}{2^n},$$
which tends to $0$ as $n$ tends to infinity. Thus, $B\in\I_d\subseteq \I$, a contradiction with $B\not\in\I$.
\end{proof}

\begin{theorem}\ 
\label{thm:TI-max}
\begin{enumerate}
\item If $\J$ is a T-maximal ideal then $\J = t(\I)$ for any maximal ideal $\I\supseteq \J$.\label{thm:TI-max-is-t-of-max}
\item There is a maximal ideal $\I$ such that $t(\I)$ is not a T-maximal ideal.
\item Any T-maximal  ideal does not have the Baire property.
\end{enumerate}
\end{theorem}

\begin{proof}
(1)
($\subseteq$)
If $A\in \J$, then $A+k\in \J\subseteq\I $ for every $k\in \Z$. Thus, $A\in t(\I)$.

($\supseteq$)
Since $\J\subseteq t(\I)$, $t(\I)$ is translation invariant and $\J$ is maximal among all translation invariant ideals,  $\J=t(\I)$.

(2)
Let $B=\{b_n:n\in\N\}$ be such that the sequence $\langle b_{n+1}-b_n: n\in\N\rangle$ is increasing and diverging to infinity.
Let $\I$ be a maximal ideal with $\N\setminus B \in \I$.
If we show that $t(\I)\cup\{B\}$ can be extended to a translation invariant ideal, 
$t(\I)$ will not be a T-maximal ideal.

To prove that $t(\I)\cup\{B\}$ can be extended to a translation invariant ideal, we have to show that 
$\N\setminus[(B+k_1)\cup\dots\cup(B+k_n)]\notin t(\I)$
for any $k_1,\dots,k_n\in\Z$.

Let $k_1,\dots,k_n\in\Z$.
Let $k\in\Z\setminus\{-k_1,\dots,-k_n\}$.
Since $|(B+l)\cap B|\leq 1$ for every $l\neq0$, 
$\left([(B+k_1)\cup\dots\cup(B+k_n)]+k\right)\cap B$ is finite.
Since $\N\setminus B\in\I$, 
$[(B+k_1)\cup\dots\cup(B+k_n)]+k\in \I$.
Then 
$\N\setminus ([(B+k_1)\cup\dots\cup(B+k_n)]+k)\notin \I$, 
and consequently 
$(\N\setminus [(B+k_1)\cup\dots\cup(B+k_n)])+k\notin \I$ which means that
$\N\setminus[(B+k_1)\cup\dots\cup(B+k_n)]\notin t(\I)$.

(3)
Follows from (1) and Proposition~\ref{prop:intersection-of-maximal-ideals}(\ref{prop:intersection-of-maximal-ideals-BP}).
\end{proof}

\begin{theorem}
\label{thm:T-max-is-not-maximal-ideal-limit}
Let $\delta$ be defined as in the proof of Theorem~\ref{thm:nonTAD-ideal-with-rich-TA-density}.
Then $\cZ_\delta$ is not a T-maximal ideal.
\end{theorem}

\begin{proof}
It follows from Theorems~\ref{thm:ideal-lim-density-is-not-t-of-max} and 
\ref{thm:TI-max}(\ref{thm:TI-max-is-t-of-max}).
\end{proof}

\bigskip

The following diagram summarizes all (un)known relationships between translation invariant  ideals $\I$ and existence of rich translation invariant  abstract upper density with $\I=\cZ_\delta$.

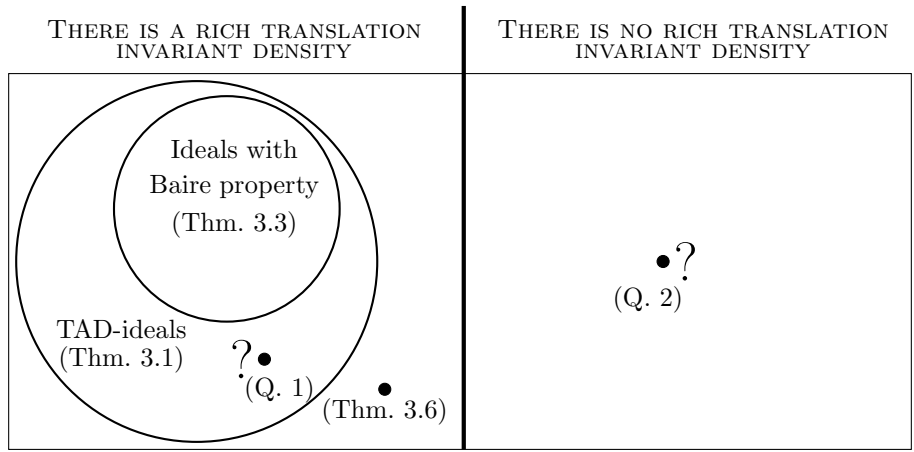
\begin{center}
\begin{tikzpicture}
\draw (-6,-2.5) rectangle (6,2.5);

\draw[ultra thick] (0.05,-2.5) -- (0.05,3.4);
\node at (-3,3.1) {\textsc{There is a rich translation}};
\node at (-3,2.8) {\textsc{invariant density}};
\node at (3.1,3.1) {\textsc{There is no rich translation }};
\node at (3.1,2.8) {\textsc{invariant density}};

\draw[thick]  (-3.5,0) circle (2.4cm);
\node at (-4.5,-0.9) {TAD-ideals};
\node at (-4.5,-1.3) {(Thm.~\ref{thm:I-TAD-of-cardinality-continuum-implies-abstract-upper-density})};

\draw[thick]  (-3.1,0.7) circle (1.5cm);
\node at (-3,1.5) {Ideals with };
\node at (-3,1) {Baire property};
\node at (-3,0.5) {(Thm.~\ref{thm:TA-rich-abstract-upper-density})};

\draw[fill] (-2.6,-1.3) circle (0.08cm);
\node at (-2.9,-1.3) {{\Huge?}};
\node at (-2.4,-1.7) {(Q.~\ref{q:TAD-ideal-withoutBP})};

\draw[fill] (-1,-1.7) circle (0.08cm);
\node at (-1,-2) {(Thm.~\ref{thm:nonTAD-ideal-with-rich-TA-density})};

\draw[fill] (2.7,0) circle (0.08cm);
\node at (3,0) {{\Huge ?}};
\node at (2.5,-0.5) {(Q.~\ref{q:Tideal-without-density})};
\end{tikzpicture}
\end{center}

\section*{Acknowledgment}

The authors would like to thank Piotr Szuca for fruitful discussion. 


\bibliographystyle{amsplain}
\bibliography{paper}

\end{document}